\documentclass[11pt,twoside,a4paper]{article}
\usepackage{amsmath,amssymb,amsthm}
\usepackage{hyperref}
\usepackage{graphicx,psfrag}

\newtheorem{thm}{Theorem}[section]
\newtheorem{lem}[thm]{Lemma}
\newtheorem{pro}[thm]{Proposition}
\newtheorem{cor}[thm]{Corollary}

\theoremstyle{definition}

\theoremstyle{remark}
\newtheorem{rem}[thm]{Remark}

\newcommand{\R}{\mathbb{R}}

\newcommand{\N}{\mathbb{N}}
\newcommand{\C}{\mathbb{C}}

\newcommand{\cA}{\mathcal{A}}

\newcommand{\cC}{\mathcal{C}}
\newcommand{\cD}{\mathcal{D}}

\newcommand{\cH}{\mathcal{H}}

\newcommand{\cP}{\mathcal{P}}

\newcommand{\al}{\alpha}

\newcommand{\ga}{\gamma}
\newcommand{\Ga}{\Gamma}
\newcommand{\de}{\delta}
\newcommand{\De}{\Delta}
\newcommand{\ep}{\varepsilon}
\newcommand{\om}{\omega}

\newcommand{\si}{\sigma}

\newcommand{\la}{\lambda}
\newcommand{\La}{\Lambda}
\renewcommand{\phi}{\varphi}

\newcommand{\dist}{\operatorname{dist}}
\newcommand{\diam}{\operatorname{diam}}

\newcommand{\Lip}{\operatorname{Lip}}

\newcommand{\tr}{\operatorname{Tr}}

\newcommand{\dom}{\operatorname{dom}}
\newcommand{\dens}{\operatorname{dens}}
\newcommand{\loc}{\operatorname{loc}}
\newcommand{\dil}{\operatorname{dil}}
\newcommand{\supp}{\operatorname{supp}}
\newcommand{\mes}{\operatorname{mes}}

\newcommand{\es}{\emptyset}
\newcommand{\set}[2]{\{#1:\,\text{#2}\}}
\newcommand{\sm}{\setminus}
\newcommand{\sub}{\subset}

\newcommand{\ov}{\overline}
\newcommand{\wt}{\widetilde}

\begin{document}

\title{Spectral geometries on a compact metric space}
\author{Sergei Buyalo\footnote{This work is supported by RFBR Grant
99-01-00104 and by the Program of the Presidium of the Russian
Academy of Sciences №01 'Fundamental Mathematics and its Applications' under
Grant PRAS-18-01}}

\date{}
\maketitle

 \begin{abstract} The notion of a spectral geometry on a compact metric
space
$X$
is introduced. This notion serves as a discrete approximation of
$X$
motivated by the notion of a spectral triple from noncommutative geometry.
A set of axioms charaterizing spectral geometries is given. Bounded deformations
of spectral geometries are studied and the relationship between the dimension of a
spectral geometry and more traditional dimensions of metric spaces is investigated.
 \end{abstract}

\section{Introduction}

The notion of a spectral geometry on a compact metric space
$X$
is introduced in \cite{Bu} as a tool of discrete approximation of
$X$
motivated by the notion of a spectral triple from noncommutative geometry,
see \cite{Co}. Each spectral geometry
$M$
on
$X$
defines a Connes metric
$d_M$
on
$X$.
In sect.~\ref{sect:axioms_spectral}, we study relations between
$d_M$
and the initial metric 
$d$
of
$X$,
and also give a set of simple axioms which charaterize spectral geometries.
In particular, we prove the existence of spectral geometries 
$M$
on
$X$
such that identity map
$(X,d)\to (X,d_M)$
is a bi-Lipshitz homeomorphism with the Lipshitz constant arbitrarily close to 1.

Spectral geometries seem to be a convenient tool for construction of metrics with
prescribed properties. A corresponding construction called a {\em bounded deformation}
of a spectral geometry is considered in sect.~\ref{sect:bounded_deformations}. 
There, we study relations between different types of convergence on the space 
of deformations and topologies on the space of metrics.

In the last section~\ref{sect:dimension_spectral_geometry}, we study relations 
between the dimension of a spectral geometry
$M$
on
$X$
introduced in \cite{Bu} and traditional notions such as the Hausdorff dimension,
the upper and lower box-counting dimensions, the Mikowsky-Bouligand dimension. 

\section{Axioms of spectral geometries}
\label{sect:axioms_spectral}

\subsection{The main construction}
\label{subsect:main_construction}

Assume a set 
$B\sub X^2\sm\De$, $\De=\set{(x,x)}{$x\in X$}$ 
is the diagonal, possesses the following properties

\begin{itemize}
 \item [(a)] $B$
is symmetric, i.e., 
$(y,y')\in B\Leftrightarrow (y',y)\in B$;
 \item [(b)] the set 
$B_t=\set{(y,y')\in B}{$|yy'|\ge t$}$
is finite for each
$t>0$.
\end{itemize}

Then we define a spectral triple
$M(B)=\{\cA,\cH,ds\}$
as follows. The Hilbert space
$\cH$
consists of functions
$\xi:B\to\C$
satisfying the condition
$$\|\xi\|^2=\sum_{(y,y')\in B}|\xi(y,y')|^2<\infty.$$
The algebra
$\cA=C(X)$
of all continuous functions
$X\to\C$
is represented by the multiplication operators
$$(\pi(a)(\xi)(y,y')=a(y)\xi(y,y')$$
for all
$a\in\cA$, $\xi\in\cH$, $(y,y')\in B$.
Clearly, this represention is involutive,
$\pi(a^\ast)=\pi(a)^\ast$,
and
$$\|\pi(a)\|=\sup_{y\in Y}|a(y)|\le\max_{x\in X}|a(x)|,$$
where
$Y$
is the projection of 
$B$
on
$X$
(by (a) it is the same whether one takes the projection on the first or on the 
second factor).

The unit length operator
$ds$
is defined by
$$ds(\xi)(y,y')=|yy'|\xi(y',y).$$
This operator is injective and self-adjoint. Its eigenvalues are
$\pm|yy'|$, $(y,y')\in B$.
It follows from (b) that 
$ds$
is compact.

The operator
$\cD=ds^{-1}$
is defined on
$$\dom\cD=\set{\xi\in\cH}{$\sum_{(y,y')\in B}\|\xi(y,y')\|^2/|yy'|^2<\infty$}$$
and called the {\em Dirac} operator. The set 
$\dom\cD$
contains all functions 
$\xi\in\cH$
with finite support and whence is dense in
$\cH$.
Furthermore,
$\dom\cD=\dom\cD^\ast$
and
$\cD=\cD^\ast$.
Thus
$\cD$
is an unbounded self-adjoint operator with discrete spectrum
$$\set{\pm|yy'|^{-1}}{$(y,y')\in B$}\sub\R.$$
The set 
$\dom\cD$
is invariant under every
$\pi(a)$, $a\in\cA$,
and we have 
$$([\cD,\pi(a)]\xi)(y,y')=\frac{a(y')-a(y)}{|yy'|}\xi(y',y).$$
Then
$$\|[\cD,\pi(a)]\|=\sup_{(y,y')\in B}\frac{|a(y)-a(y')|}{|yy'|}\le\Lip(a).$$
Thus the operator
$[\cD,\pi(a)]$
can be extended to a bounded operator on
$\cH$.
Now the formula
$$|xx'|_M=\sup\set{|a(x)-a(x')|}{$\|[\cD,\pi(a)]\|\le 1$}$$
defines the Connes metric
$(x,x')\mapsto|xx'|_M$
on
$X$.

\subsection{Density of a spectral triple}
\label{subsect:density_sptriple}

We use the metric
$|zz'|=\max\{|xx'|,|yy'|\}$
on
$X^2=X\times X$,
where
$z=(x,y)$, $z'=(x',y')$.
We denote by
$|z|=|xy|$
the distance between
$x$, $y\in X$.
The {\em density} of the subset
$B\sub X^2$
is defined by
$$\dens(B)=\inf_{z\in X^2\sm\De}\sup_{b\in B}|z|/|zb|.$$
The number
$\dens(M):=\dens(B)$
is called the {\em density} of the spectral triple
$M(B)$.
If the set 
$X$
is infinite, then it follows from (b) that 
$\dens(M)<\infty$.
If
$\dens(M)>0$,
then for each
$\ep\in(0,\dens(M))$
and 
$(x,x')\in X^2\sm\De$
there is a point
$(y,y')\in B$
with
$$|xy|,|x'y'|\le\al|xx'|,$$
where
$\al=1/(\dens(M)-\ep)$.
Thus the set 
$Y$,
the projection of
$B$
on
$X$,
is dense in the set of nonisolated points of
$X$.
If
$\dens(M)>1$,
then
$Y$
is everywhere dense in
$X$,
in particular,
$$\|\pi(a)\|=\max_{x\in X}|a(x)|$$
for each function 
$a\in\cA$.
In the case
$\dens(M)>2$,
one can take 
$\al=1/\dens(M)$
in the inequality above. This easily follows from (b).

\begin{lem}\label{lem:sup_est_below} Assume that the set 
$B\sub X^2\sm\De$
has the density
$D=\dens(B)>2$.
Then for every function
$a\in\cA$
we have
$$\sup_{(y,y')\in B}\frac{|a(y)-a(y')|}{|yy'|}\ge\frac{D-2}{D+2}\cdot\Lip(a).$$
\end{lem}

\begin{proof} Fix
$\ep\in(0,D-2)$.
Then
$\al:=1/(D-\ep)<1/2$,
and for each point
$(x,x')\in X^2\sm\De$
there is a point
$(y,y')\in B$
with
$|xy|$, $|x'y'|\le\al|xx'|$.
In particular, this is true for 
$(x,x')\in X^2\sm\De$
with
$|a(x)-a(x')|/|xx'|\ge L-\ep$,
where
$L=\Lip(a)$.
Then
$|yy'|\le(1+2\al)|xx'|$,
and we have

\begin{align*}
 |a(x)-a(x')|/|xx'|&\le(|a(y)-a(y')|/|yy'|)(1+2\al)+L(|xy|+|x'y'|)/|xx'|\\
                   &\le(1+2\al)|a(y)-a(y')|/|yy'|+2\al L,
\end{align*}
which implies
$|a(y)-a(y')|/|yy'|\ge L\cdot(1-2\al)/(1+2\al)-\ep/(1+2\al)$.
The claim follows from that the 
$\ep>0$
was chosen arbitrary.
\end{proof}

\begin{pro}\label{pro:est_two_side} Let
$M=M(B)$
be a spectral triple of density
$D=\dens(M)>2$
on a compact metric space 
$X$.
Then for every
$x$, $x'\in X$
we have
$$|xx'|\le|xx'|_M\le\frac{D+2}{D-2}|xx'|.$$
Moreover, for each
$(y,y')\in B$
we have
$|yy'|=|yy'|_M$,
where
$|\ |_M$
is the Connes metric associated with
$M$.
\end{pro}

\begin{proof} If
$a\in\cA$
and
$\Lip(a)\le 1$,
then
$\|[\cD,\pi(a)]\|\le 1$,
where
$\cD=ds^{-1}$.
Choosing
$a=a_x:X\to\R$, $a_x(x')=|xx'|$,
we obtain
$|xx'|_M\ge|xx'|$.
By Lemma~\ref{lem:sup_est_below}
we have
$\Lip(a)\le(D+2)/(D-2)$.
In the case
$\|[\cD,\pi(a)]\|\le 1$
this gives 
$$|xx'|_M\le\frac{D+2}{D-2}|xx'|.$$
If
$(y,y')\in B$
and
$\|[\cD,\pi(a)]\|\le 1$,
then
$|a(y)-a(y')|\le\|[\cD,\pi(a)]\|\cdot|yy'|\le|yy'|$,
whence
$|yy'|_M\le|yy'|$.
\end{proof}

\begin{cor}\label{cor:reapeted_connes_metric} Let
$M=M(B)$
be a spectral triple with density
$D=\dens(M)>2$
on a compact metric space
$X$.
Then
$M$
is a spectral triple
$M'$
with density
$D'=\dens(M')$
such that
$$D\frac{D-2}{D+2}\le D'\le D\frac{D+2}{D-2}$$
on the compact metric space
$(X,\dist_M)$,
where
$\dist_M$
is the Connes metric associated with
$M$.
Furthermore,
$\dist_{M'}=\dist_M$.
\end{cor}

\begin{proof} Using Proposition~\ref{pro:est_two_side}, we obtain
$$\frac{D-2}{D+2}\frac{|z|}{|zb|}\le\frac{|z|_M}{|zb|_M}\le\frac{D+2}{D-2}\frac{|z|}{|zb|}$$ 
for each
$z\in X^2\sm\De$, $b\in B$.
This immediately implies the required estimate for 
$\dens(M')$.
Since
$|yy'|_M=|yy'|$
for every
$(y,y')\in B$, 
the unit length and Dirac operators for the triples
$M$
and
$M'$
coincide. Hence
$\dist_{M'}=\dist_M$.
\end{proof}

The property
$\dist_{M'}=\dist_M$
for Connes metrics is similar to the well known property of intrinsic metrics
to be invariant under the standard procedure of getting an intrinsic metric from
a given one.

\subsection{Local density}
\label{subsect:local_dens}

Let
$\dens(B,U_t)$
be the density of the set 
$B\sub X^2$
in the neighborhood
$U_t=\set{z\in X^2}{$|z|\le t$}$
of the diagonal
$\De$.
This function does not increase in
$t$,
and thus is there the limit
$$\dens_{\loc}(B)=\lim_{t\to 0}\dens(B,U_t),$$
which is called the {\em local density} of
$B$.
Clearly,
$\dens_{\loc}(B)\ge\dens(B)$.
Thus in contrast to the density
$\dens(M)$,
the local density
$\dens_{\loc}(M)=\dens_{\loc}(B)$
of a spectral triple
$M=M(B)$
might be infinite even of an infinite
$X$
(see sect.~\ref{subsect:exist_spgeom_dens}). 

If
$X$
is a space with intrinsic metric, then for each Lipshitz function
$a$
on
$X$
we have
$$\Lip(a)=\sup_{x\in X}\dil_x(a),$$
where
$\dil_x(a)=\lim_{\ep\to 0}\Lip(a|D_\ep(x))$
is the dilatation of
$a$
at the point
$x$, $D_\ep(x)=\set{x'\in X}{$|xx'|\le\ep$}$.
In that case, Lemma~\ref{lem:sup_est_below} and Proposition~\ref{pro:est_two_side}
remain true for 
$D=\dens_{\loc}(M)$.
In particular, we have

\begin{pro}\label{pro:intrinsic_local} Let
$M=M(B)$
be a spectral triple of infinite local density on a space
$X$
with intrinsic metric. Then its Connes metric coincides with the metric of
$X$,
$$|xx'|_M=|xx'|$$
for every
$x$, $x'\in X$.
\qed
\end{pro}

One can slightly modify the main construction by introducing additionally
a symmetric function
$m:B\to\N$, $m(y,y')=m(y',y)$
for all
$(y,y')\in B$, 
a {\em multiplicity} function. The space
$\cH$
is then defined as the space of sections of the bundle
$p:\C^m\to B$, 
satisfying the condition
$$\|\xi\|^2=\sum_{(y,y')\in B}\|\xi(y,y'\|^2<\infty,$$
where the fiber over a point
$(y,y')\in B$
is a finite dimensional Hilbert space
$p^{-1}(y,y')\simeq\C^{m(y,y')}$.
The definitions of the represention
$\pi$
and the operator
$ds$
remain the same as above. The fact that
$ds$
is well defined follows from that the function
$m$
is symmetric. We denote by
$M(B,m)$
so modified spectral triple
$\{\cA,\cH,ds\}$.
Obviously,
$M(B)=M(B,1)$.
It is easy to see that Propositions~\ref{pro:est_two_side} and \ref{pro:intrinsic_local},
and Corollary~\ref{cor:reapeted_connes_metric} are true for triples
$M(B,m)$
as well.

\subsection{Axioms of spectral geometry}
\label{subsect:axioms_spectral_geometry}

We give here a set of axioms which leads to spectral triple of form 
$M(B,m)$.

Recall (see \cite[Chapter IV]{N}) that a spectral measure on a locally compact space
$T$
is any operator-valued function
$P$,
defined on the Borel subsets of
$T$
and having the properties

\begin{itemize}
 \item given a Borel subset
$A\sub T$,
the operator
$P(A)$
is an orthogonal projector in a fixed Hilbert space
$\cH$;
 \item given
$\xi\in\cH$,
the function
$(P(A)\xi,\xi)$
is a Borel measure on
$T$.
\end{itemize}
In particular,
$P(A_1\cup A_2)=P(A_1)+P(A_2)$
and
$P(A_1)P(A_2)=0$
for 
$A_1\cap A_2=\es$.

A {\em spectral geometry} on a compact metric space
$X$
is given by an involutive represention
$\pi$
of the algebra
$\cA$
of Lipshitz functions on
$X$
in a Hilbert space
$\cH$
and an injective compact self-adjoint operator
$ds$
on
$\cH$
such that the following conditions (i)--(iv) are satisfied  

\begin{itemize}
 \item [(i)] the represention
$\pi$
is the restriction on
$\cA$
of an involutive represention
$\wt\pi$
of the 
$C^\ast$-algebra
$\cC(X^2)$
of continuous functions on
$X^2$,
where a function
$a\in\cA$
is considered as the continuous function
$a(x,x')=a(x)$
on
$X^2$.
\end{itemize}

According to a classical result of Dunford (see \cite[Chapter IV]{N}),
the knowledge of the represention
$\wt\pi$
is equivalent to the knowledge of a spectral measure
$P$
on
$X^2$
such that
$$\wt\pi(a)=\int_{X^2}a(x,x')dP(x,x')$$
for every
$a\in\cC(X^2)$,
in particular,
$\pi(a)=\int_{X^2}a(x)dP(x,x')$
for every
$a\in\cA$.

\begin{itemize}
 \item [(ii)] $ds\circ P(A)=P(\ga(A))\circ ds$
for every Borel subset
$A\sub X^2$,
where
$\ga(x,x')=(x',x)$
for 
$(x,x')\in X^2$.
\end{itemize}

It follows from (ii) that the operator
$ds^2$
commutes with
$P$,
in particular,
$[ds^2,\pi(a)]=0$
for every
$a\in\cA$.

Let
$L_\la$
be the eigenspace of
$ds^2$
corresponding to an eigenvalue
$\la^2$, $A_\la=A_{-\la}$
the union of all open in
$X^2$
subsets
$A$
for which 
$P(A)$
is a projector on a subspace orthogonal to
$L_\la$
(i.e. lying in the orthogonal complement
$L_\la^\perp$
to
$L_\la$).
Then as easy to see (Lemma~\ref{lem:ds_eigenspace}),
$P(A_\la)$
is also a projector on a subspace orthogonal to
$L_\la$,
and
$A_\la$
is a (unique) maximal open subset with this property.

\begin{itemize}
 \item [(iii)] $P(A_\la)$
is the projector onto orthogonal complement
$L_\la^\perp$
of
$L_\la$
for each eigenvalue
$\la$
of
$ds$.
\end{itemize}

We put
$\La_\la=X^2\sm A_\la$.

\begin{itemize}
 \item [(iv)] $|yy'|=|\la|$
for every
$(y,y')\in\La_\la$.
\end{itemize}

The set 
$B=\cup_\la\La_\la$,
where
$\la$
runs over all eigenvalues of
$ds$,
is called the {\em support} of the spectral geometry
$\{\cA,\cH,ds\}$.
This terminology is justified by the fact that the closure 
$\ov B$
coincides with the support of the operator
$ds$
(see Corollary~\ref{cor:supp_spgeom}).

\begin{pro}\label{pro:spectral_triple_spgeom} Let
$X$
be a compact metric space, a subset
$B\sub X^2\sm\De$
satisfies conditions (a), (b) of the main construction,
a function
$m:B\to\N$
is symmetric. Then the spectral triple
$M(B,m)=\{\cA,\cH,ds\}$
is a spectral geometry on
$X$.
\end{pro}

\begin{proof} The represention
$\pi$
is the restriction on
$\cA$
of the represention
$\wt\pi$,
$$(\wt\pi(a)\xi)(y,y')=a(y,y')\xi(y,y'),\ (y,y')\in B,$$
of 
$C^\ast(X^2)$-algebra
in
$\cH$.

Given
$(y,y')\in B$,
we define the projector
$P(y,y')$
as
$$(P(y,y')\xi)(z,z')=\begin{cases}
                      \xi(y,y'),\ &z=y,z'=y'\\
                      0, &\textrm{otherwise},
                     \end{cases}
$$
$\xi\in\cH$.
Then the spectral measure
$P$
associated with
$\wt\pi$
is defined by
$$P(A)=\sum_{(y,y')\in A\cap B}P(y,y')$$
for every Borel
$A\sub X^2$.
It follows from the definitions of operators
$ds$
and
$P(y,y')$
that
$ds\circ P(y,y')=P(y',y)\circ ds$
for every
$(y,y')\in B$
which implies condition (b) of the spectral geometry.

The set 
$\La_\la=\set{(y,y')\in B}{$|yy'|=|\la|$}$
is finite according to (b), and 
$P(\La_\la)$
is the projector on the 
$\la^2$-eigenspace
$L_\la$
of
$ds^2$
for each eigenvalue
$\la$
of
$ds$.
Thus conditions (iii) and (iv) from definition of a spectral
geometry are fulfilled.
\end{proof}

The main result of this section is that any spectral geometry can
be obtained by using the main construction (and taking into account
the multiplicity function
$m$
of eigenvalues of the unit length operator).

\begin{thm}\label{thm:spectral_triple_spgeom} Let
$\{\cA,\cH,ds\}$
be spectral geometry on a compact metric space
$X$.
Then its support
$B\sub X^2\sm\De$
satisfies conditions (a), (b), and there exists a symmetric function
$m:B\to\N$,
for which 
$\{\cA,\cH,ds\}$
is isomorphic to the spectral geometry
$M(B,m)=\{\cA,\cH',ds'\}$
by an isometry 
$U:\cH\to\cH'$
conjugating the representions of
$\cA$
and the unit length operators
$ds$, $ds'$. 
\end{thm}

The proof is based on the following lemma.

\begin{lem}\label{lem:ds_eigenspace} Assume that a spectral triple
$\{\cA,\cH,ds\}$
on a compact metric space
$X$
satisfied axioms (i), (ii) of spectral geometry. Let
$L_\la$
be the eigenspace of
$ds^2$
corresponding to an eigenvalue
$\la^2$.
Then there exists a unique maximal open subspace
$A_\la\sub X^2$,
for which 
$P(A_\la)$
is the projector on a subspace orthogonal to
$L_\la$.
Furthermore, its complement
$\La_\la=X^2\sm A_\la$
is finite.
\end{lem}

\begin{proof} Let
$\{A_\si\}$
be the collection of all open subsets in
$X^2$
for which 
$P(A_\si)$
is a projector on a subspace orthogonal to
$L_\la$, $A_\la=\cup_\si A_\si$.
We show that 
$P(A_\la)$
is also a projector on a subspace
$Q_\la$
which is orthogonal to
$L_\la$.
Let
$P_\la$
be the projector on
$L_\la$.
Then, since the spectral measure
$P$
commutes with
$P_\la$,
the operator
$P(A_\la)\circ P_\la$
is a projector on
$Q_\la\cap L_\la$,
Thus if the assertion is not true, then there exists a vector
$\xi\in L_\la$, $\|\xi\|=1$,
for which
$P(A_\la)\xi=\xi$.
Then the Borel measure
$\mu=(P(A_\la)\xi,\xi)$
has the properties:
$\mu(A_\si)=0$
for all
$\si$
and
$\mu(A_\la)=1$.
This, obiously, is impossible, because for every
$\ep>0$
one can find a compact subset
$K\sub A_\la$
with 
$\mu(K)\ge 1-\ep$.
On the other hand, one can cover
$K$
by a finite number of subsets from
$\{A_\si\}$
and consequently
$\mu(K)=0$.

Therefore,
$A_\la$
is a maximal open subset for which
$P(A_\la)$
is the projector on a subspace orthogonal to
$L_\la$.
Obviously, it is unique with this property. It remains to show that 
its complement
$\La_\la$
is finite.

The operator
$ds^2$
is compact, thus the space
$L_\la$
is finite dimensional. We show that 
$\La_\la$
consists of at most
$\dim L_\la(=\dim_\C L_\la)$
points. If this is not the case, then choosing disjoint
neighborhoods
$U_i$
of points
$z_1\dots z_n\in\La_\la$, $n>\dim L_\la$,
we find that at least one of the projectors
$P(U_i)$
is the projector on a subspace orthogonal to
$L_\la$.
The set 
$A_\la'=A_\la\cap U_i$
is open in
$X^2$
and contains
$A_\la$
as a proper subset. Furthermore,
$P(A_\la')=P(A_\la)+P(U_i\cap\La_\la)$
is the projector on a subspace orthogonal to
$L_\la$.
This contradicts the fact that 
$A_\la$
is maximal.
\end{proof}

\begin{proof}[Proof of Theorem~\ref{thm:spectral_triple_spgeom}] By 
axiom~(iv), given
$(y,y')\in B$
there exists an eigenvalue
$\la$
of
$ds$
with
$|yy'|=|\la|$.
Since
$ds$
is injective, it follows that
$B\sub X^2\sm\De$.
The subspace
$L_\la$
is invariant for 
$ds$,
thus it follows from (ii) that
$A_\la$
is symmetric. Thus
$B$
is symmetric,
$\ga(B)=B$.
Using (iv), we obtain that
$B_t=\set{(y,y'\in B}{$|yy'|\ge t$}$
coincides with
$\cup_{|\la|\ge t}A_\la$,
where the union is taken over all eigenvalues
$\la$
of
$ds$
with
$|\la|\ge t$.
By Lemma~\ref{lem:ds_eigenspace}, each
$\La_\la$
is finite. Then
$B_t$
is finite for each
$t>0$
since
$ds$
is compact. Therefore,
$B$
satisfies conditions (a), (b) of the main construction.

By Lemma~\ref{lem:ds_eigenspace} and axiom~(iii), the operator
$P(y,y')$
is a nonzero projector on a subspace
$L_{y,y'}\sub L_\la$
for every
$(y,y')\in \La_\la$.
We define a function
$m:B\to\N$
by
$m(y,y')=\dim L_{y,y'}$.
It follows from (ii) that
$$ds(L_{y,y'})=ds\circ P(y,y')(\cH)=P(y,y')\circ ds(\cH)=L_{y',y},$$
thus
$m$
is symmetric.

Using injectivity of
$ds$
and axiom~(i), now it is easy to construct an isometry
$U:\cH\to\cH'$
establishing an isomorphism between
$\{\cA,\cH,ds\}$
and
$M(B)=\{\cA,\cH',ds'\}$.
\end{proof}

The {\em support} of an operator
$A$
in
$\cH$
is defined as the complement to the union of all open subsets
$U\times U'\sub X^2$,
for which
$\pi(a)\circ A\circ\pi(a)=0$
for all
$a$, $a'\in\cA$
with
$\supp(a)\sub U$, $\supp(a')\sub U'$.

\begin{cor}\label{cor:supp_spgeom} Let
$M=\{\cA,\cH,ds\}$
be a spectral geometry on a compact metric space
$X$. 
Then the support of
$ds$
coincides with the closure of its support
$B$, $\supp ds=\ov B$.
\end{cor}

\begin{proof} By Theorem~\ref{thm:spectral_triple_spgeom}, the spectral
geometry
$M$
coincides with
$M(B,m)$,
where
$m$
is the multiplicity function. For every
$a$, $a'\in\cA$, $\xi\in\cH$
and
$(y,y')\in B$
we have
$$(\pi(a)\circ ds\circ\pi(a)\xi)(y,y')=a(y)\cdot|yy'|\cdot a'(y')\cdot\xi(y',y).$$
Thus
$\pi(a)\circ ds\circ\pi(a')=0\Leftrightarrow a(y)\cdot a'(y')=0$
for all
$(y,y')\in B$.
Hence the claim.
\end{proof}

\subsection{Existence of spectral geometries with nonzero density}
\label{subsect:exist_spgeom_dens}

Given a compact metric space
$X$,
we show how to construct a spectral geometry on
$X$
with arbitrary large density. By Theorem~\ref{thm:spectral_triple_spgeom},
it suffices to find a subset
$B\sub X^2\sm\De$
satisfying conditions (a), (b) of the main construction and
having the density bigger than a given number.

\begin{pro}\label{pro:large_density} Given
$D\in (2,\infty]$,
let
$\{Y_k\}$
be a sequence of finite 
$\de_k$-nets
in
$X$
such that
$\de_k\searrow 0$
as
$k\to\infty$.
We put
$B=\cup_kB(k)$,
where
$$B(k)=\set{(y,y')\in Y_k^2\sm\De}{$|yy'|\le\de_{k-1}\phi(\de_{k-1}$},$$
$\phi(t)=2D$
if
$D<\infty$,
and if
$D=\infty$,
then
$\phi(t)\ge 4$
be a function with
$\phi(t)\to\infty$
and
$t\phi(t)\searrow 0$
as
$t\to 0$.

Then the set 
$B$
satisfies conditions (a), (b) of the main construction 
and has density
$\dens(B)\ge D$,
if
$D<\infty$,
and the local density
$\dens_{\loc}(B)=\infty$,
if
$D=\infty$.
\end{pro}

\begin{proof} The set 
$B$
is obviously symmetric. Let
$t>0$.
Then for 
$(y,y')\in B_t$
we have
$(y,y')\in B(k)$
and
$t\le|yy'|\le\de_{k-1}\phi(\de_{k-1})$.
Hence
$k\le k(t)$
for all
$(y,y')\in B_t$.
Thus
$B_t$
is finite. 

One can assume that
$\diam X\le\de_1\phi(\de_1)/2$.
Given
$x\neq x'\in X$,
there exists 
$k\ge 2$
such that
$\de_k\phi(\de_k)/2\le|xx'|\le\de_{k-1}\phi(\de_{k-1})/2$.
Since
$Y_k$
is a
$\de_k$-net,
one can find
$y$, $y'\in Y_k$
with
$|xy|$, $|x'y'|\le\de_k\le 2|xx'|/\phi(\de_k)$.
Furthermore,
$|yy'|\ge|xx'|-2\de_k\ge\de_k(\phi(\de_k)-2)>0$
and
$|yy'|\le|xx'|+2\de_k\le\de_{k-1}\phi(\de_{k-1})/2+2\de_k\le\de_{k-1}\phi(\de_{k-1})$. 
Thus
$(y,y')\in B$.
Therefore, given
$z=(x,x')\in X^2\sm\De$,
we have found 
$b=(y,y')\in B$
with
$|z|/|zb|\ge\phi(\de_k)/2$.
Thus
$\dens(B)\ge D$,
if
$D<\infty$.
For 
$D=\infty$
this argument shows that a point
$z\in X^2\sm\De$
can be approximated by points of
$B$
with a prescribed arbitrarily small error, if
$z$
lies in a sufficiently small neighborhood of the diagonal
$\De$.
Thus
$\dens_{\loc}(B)=\infty$.
\end{proof}

\section{Bounded deformations of a spectral geometry}
\label{sect:bounded_deformations}

If axioms (i)--(iii) for a spectral triple are fulfilled, then
axiom~(iv) uniquely defines a spectral geometry whose unit length
operator differs from that one of the spectral triple only by 
the spectrum. Forgetting axiom~(iv) we come up with the notion
of deformation of a spectral geometry.

Let
$M=\{\cA,\cH,ds\}$
be a spectral geometry on a compact metric space
$X$.
Its {\em deformation} is given by a compact self-adjoint operator
$ds'$
on
$\cH$
of the form
$$ds'(\xi)(y,y')=\rho(y,y')\xi(y',y),$$
where
$\rho:B\to\R$
is a symmetric positive function on the support 
$B$
of
$M$.
Notation:
$ds'=ds_\rho$.
Therefore, the deformation
$ds_\rho$
of
$M$
can be identified with a positive symmetric function
$\rho:B\to\R$
which is called the {\em marked spectrum} of
$ds_\rho$.

A deformation
$ds_\rho$
is said to be {\em bounded}, if the ratio
$\rho(y,y')/|yy'|$
is bounded and separated away from zero on
$B$.

Every deformation
$ds_\rho$
of a spectral geometry
$\{\cA,\cH,ds\}$
defines the Connes metric on
$X$,
$$|xx'|_\rho=\sup\set{|a(x')-a(x)|}{$a\in\cA, \|[\cD_\rho,\pi(a)]\||\le 1$},$$
where
$\cD_\rho$
is the operator inverse to
$ds_\rho$.
The next proposition immediately follows from definitions.

\begin{pro}\label{pro:bounded_deformations_bilipschitz} Let
$ds_\rho$
be a bounded deformation of a spectral geometry
$M$
on a compact metric space
$X$.
Then its Connes metric
$|\ |_\rho$
is biLipschitz equivalent to the Connes metric
$|\ |_M$
of
$M$.
More precisely,
$$c|xx'|_M\le|xx'|_\rho\le C|xx'|_M\quad\textrm{for all}\quad x\neq x'\in X$$
and any constants
$C$, $c>0$
satisfying the condition
$c\le\rho(y,y')/|yy'|\le C$
for all
$(y,y')\in\supp M$.
\qed 
\end{pro}

\subsection{Bounded deformations of a connected spectral geometry}
\label{subsect:bounded_deformations_connected}

Let
$M=\{\cA,\cH,ds\}$
be a spectral geometry on a compact metric space
$X$, $B\sub X^2\sm\De$
its support. Given
$t>0$
we define
$\Ga_t$        
as a graph whose set of oriented edges is
$B_t=\set{b\in B}{$|b|\ge t$}$
and the vertex set 
$Y_t$
is the projection of
$B_t$
on
$X$.
The graph
$\Ga_t$
is called the {\em incidence} graph of level
$t$
of 
$M$.
It follows from the axioms that the incidence graph
$\Ga_t$
is finite and has no isolated vertices for every
$t>0$.

A spectral geometry
$M$
is said to be {\em connected} if its incidence graph is connected 
for each sufficiently small
$t>0$.
The property of a spectral geometry to be connected is not 
very much restrictive, and the construction described in
Proposition~\ref{pro:large_density} usually gives a connected
spectral geometry.

Let
$Y\sub X$
be the projection of the support
$B\sub X^2$
on the factor
$X$.
Then distinct points
$y$, $y'\in Y$
are vertices of the incidence graph
$\Ga_t$
for every sufficiently small level
$t>0$.

Let
$\rho$
be the marked spectrum of a bounded deformation
$ds_\rho$
of
$M$.
Its restriction to
$\Ga_t$
defines an intrinsic metric
$\ov\rho_t$
on each connected component of the graph
$\Ga_t$,
i.e., we first extend 
$\rho$
to a locally intrinsic metric on
$\Ga_t$,
for which the length of each edge 
$(y,y')\in B$
is equal to
$\rho(y,y')$,
and then we let
$\ov\rho_t$
be the length of a shortest path in
$\Ga_t$
between
$y$
and
$y'$
with respect to this metric. The function
$\ov\rho_t$
does not decrease in
$t$,
thus there exists a limit
$$\ov\rho(y,y')=\lim_{t\to 0}\ov\rho_t(y,y').$$
If
$(y,y')\in B$
or if the geometry
$M$
is connected, then this limit is finite. Furthermore,
$\ov\rho(y,y')=0\Leftrightarrow y=y'$,
since 
$ds_\rho$
is a bounded deformation, and
$\ov\rho$
satisfies the triangle inequality by construction.

\begin{pro}\label{pro:bounded_deformation_connected} Let
$d_\rho$
be the Connes metric associated with a bounded deformation
$\rho$
of a connected spectral geometry
$M$
on a compact metric space
$X$, $D=\dens M>2$.
Then
$d_\rho(y,y')=\ov\rho(y,y')$
for every
$y$, $y'\in Y$,
where
$Y\sub X$
is the projection of the support 
$B$
of
$M$.
\end{pro}

\begin{proof} Assume that a function
$a\in\cA$
satisfies the condition
$\|[\cD_\rho,\pi(a)]\|\le 1$.
This means that
$|a(y)-a(y')|\le\rho(y,y')$
for each point
$(y,y')\in B$.
Thus
$$|a(y)-a(y')|\le\sum_{i=0}^{k-1}|a(y_{i+1})-a(y_i)|\le\sum_{i=0}^{k-1}\rho(y_{i+1},y_i)$$
for every sequence of edges with consecutive vertices
$y_0,\dots,y_k\in Y$
between
$y=y_0$
and
$y'=y_k$.
In particular,
$|a(y)-a(y')|\le\ov\rho(y,y')$
and whence
$d_\rho(y,y')\le\ov\rho(y,y')$.

For the proof of the opposite inequality we recall that the set 
$Y$
is everywhere dense in
$X$ 
due to the condition
$D>2$
(see sect.~\ref{sect:axioms_spectral}). We fix 
$y\in Y$
and consider the function
$a_y:Y\to\R$, $a_y(y')=\ov\rho(y,y')$.
Though this function is not defined on all over
$X$
and a priori is not necessarily Lipshitz on
$Y$,
nevertheless, one can apply to it the arguments of Lemma~\ref{lem:sup_est_below}
and obtain that
$$\Lip(a_y)\le\frac{D+2}{D-2}\sup_{(z,z')\in B}\frac{|a_y(z)-a_y(z')|}{|zz'|}$$
(the Lipshitz constant
$\Lip(a_y)$
is not supposed to be finite). Since
$$|a_y(z')-a_y(z)|\le\ov\rho(z,z')\le\rho(z,z')\le C|zz'|$$ 
for all
$(z,z')\in B$,
we have
$\Lip(a_y)\le C(D+2)/(D-2)<\infty$.
Thus the function
$a_y$
can be extended to a Lipshitz function on
$X$,
for which we use the same notation. Furthermore, for the operator
$\cD_\rho=ds_\rho^{-1}$
we have
$$\|[\cD_\rho,\pi(a_y)]\|=\sup_{(z,z')\in B}\frac{|a_y(z')-a_y(z)|}{\rho(z,z')}\le 1$$
and
$|a_y(y)-a_y(y')|=\ov\rho(y,y')$.
Therefore,
$d_\rho(y,y')\ge\ov\rho(y,y')$.
\end{proof}

\subsection{Bounded deformations and Hausdorff-Gromov convergence}
\label{subsect:bounded_deformations_hausdorff_gromov}
Deformations of a spectral geometry seem to be a convenient tool
for construction of metrics with prescribed properties. Thus it
is important to understand relations between different topologies
on the sets of deformations and Connes metrics associated with
these deformations. We start with the compact-open topology on the first set
and Hausdorff-Gromov topology on the second one.

The {\em compact-open} topology on the set 
$\Xi$
of bounded deformations is defined as the compact-open topology on 
the set of maps of form 
$1/\rho:B\to\R$, $\rho\in\Xi$.
In other words, the base of neighborhoods of a marked spectrum
$\rho_0\in\Xi$
with respect to this topology consists of sets
$$U_{t,\ep}(\rho_0)=\set{\rho\in\Xi}
  {$\left|\frac{1}{\rho(y,y')}-\frac{1}{\rho_0(y,y')}\right|<\ep\quad\textrm{for all}\quad 
  (y,y')\in B_t$},$$
$t$, $\ep>0$.

The condition that deformations are close with respect to compact-open
topology is too weak for the associated Connes metrics to be close even
in such rough topology as the Hausdorff-Gromov one. Nevertheless,
they are close under some additional restrictions.

\begin{thm}\label{thm:compact_open_hausdorff_gromov} Let
$M=\{\cA,\cH,ds\}$
be a connected spectral geometry of density
$D>2$
on a compact metric space
$X$.
If a sequence
$\{\rho_i\}$
of its deformations converges to
$\rho_0\in\Xi$
in the compact-open topology and
$\rho_0(y,y')\le\rho_i(y,y')\le C|yy'|$
for all
$(y,y')\in\supp M$, $i\ge 1$,
and some constant
$C>0$,
then the associated Connes metrics
$d_i=d_{\rho_i}$
on
$X$
Hausdorff-Gromov converge to the Connes metric
$d_0$
associated with 
$\rho_0$.
\end{thm}

We first briefly describe the idea of the proof. By the condition on the density
$D$
of
$M$,
the vertex set 
$Y_t$
of its incidence graph is a
$\de(t)$-net 
in
$X$
with
$\de(t)\le Dt/(D-2)$
for every level
$t>0$.
This is established in Lemma~\ref{lem:bounding_radius}. Thus it suffices to prove
the uniform convergence of the Connes metrics on sets
$Y_t$.
By Proposition~\ref{pro:bounded_deformation_connected},
we know how to find distances in the Connes metric
$d_\rho$
associated with a bounded deformation
$\rho\in\Xi$
between points of
$Y_t$.
The convergence
$\rho_i\to\rho$
in the compact-open topology implies the convergence of
intrinsic metrics
$(\ov\rho_i)_t\to (\ov\rho_0)_t$ 
of the incidence graphs of level
$t>0$.
The required convergence
$\ov\rho_i\to\ov\rho_0$
then follows from the condition
$\rho_i\ge\rho_0$.

The number
$\de(Y)=\inf\set{\de>0}{$X\sub\cup_{y\in Y}D_\de(y)$}$,
where
$D_r(x)$
is the ball
$\set{x'\in X}{$|xx'|\le r$}$,
is called the {\em bounding radius} of
$Y\sub X$.

\begin{lem}\label{lem:bounding_radius} Assume that a subset
$B\sub X^2$
is symmetric and has the density
$D>2$.
For
$t>0$
let
$Y_t\sub X$
be the projection on
$X$
of
$B_t=\set{(y,y')\in B}{$|yy'|\ge t$}$,
$\de(t)$
the bounding radius of
$Y_t$.
Then
$$\de(t)\le\frac{D}{D-2}\cdot t.$$
\end{lem}

\begin{proof} Assume that 
$\de(t)>Dt/(D-2)$.
Since 
$X$
is compact, there exists
$x\in X$
with
$|xY_t|=\de(t)$
and respectively
$x'\in Y_t$
with
$|xx'|=\de(t)$.
We fix 
$\ep\in(0,D-2)$
such that
$\de(t)/t\ge(D-\ep)/(D-2-\ep)$,
and approximate the points
$x$, $x'$
by a pair
$(y,y')\in B$
with
$|xy|$, $|x'y'|\le\al|xx'|$,
where
$\al=1/(D-\ep)<1/2$.
Then
$(1-2\al)\de(t)\ge t$
by the choice of
$\ep$,
and
$|yy'|\ge(1-2\al)\de(t)\ge t$.
Thus
$(y,y')\in B_t$
and
$y\in Y_t$.
Furthermore,
$|xy|\le\al\de(t)<\de(t)$,
which contradicts the condition
$|xY_t|=\de(t)$. 
\end{proof}

\begin{proof}[Proof of Theorem~\ref{thm:compact_open_hausdorff_gromov}] Let
$\de>0$.
By Lemma~\ref{lem:bounding_radius}, the set
$Y_t$
is a finite
$\de$-net 
in
$X$
for all sufficiently small
$t>0$.
It follows from the condition of the theorem that
$$0<c\le\rho_i(y,y')/|yy'|\le C$$
for all
$(y,y')\in B=\supp M$
and
$i\ge 0$.
Thus in view of Proposition~\ref{pro:bounded_deformations_bilipschitz} one can
assume that 
$Y_t$
is a 
$\de$-net 
also with respect to the Connes metrics
$d_0$, $d_i$, $i\ge 1$,
on
$X$.
Therefore, it suffices to show that
$d_i(y,y')\to d_0(y,y')$
as
$i\to\infty$
for every
$y$, $y'\in Y_t$. 

By Proposition~\ref{pro:bounded_deformation_connected}, we have
$d_i(y,y')=\ov\rho_i(y,y')$, $i\ge 0$,
where
$$\ov\rho_i(y,y')=\lim_{\tau\to 0}(\ov\rho_i)_\tau(y,y')$$
and
$(\ov\rho_i)\tau$
is an intrinsic metric on the incidence graph
$\Ga_\tau$,
constructed by the restriction
$\rho_i|\Ga_\tau$.

It follows from the condition 
$\rho_0\le\rho_i$, $i\ge 1$,
that
$\ov\rho_0\le\ov\rho_i$
and hence
$$\ov\rho_0(y,y')\le\liminf_{i\to\infty}\ov\rho_i(y,y').$$

To prove the opposite inequality, we approximate the distance
$d_0(y,y')=\ov\rho_0(y,y')$
by the distances
$(\ov\rho_0)_\tau(y,y')$
as
$\tau\to 0$,
and for every fixed 
$\tau>0$
the distance
$(\ov\rho_0)_\tau(y,y')$
by the distances
$(\ov\rho_i)_\tau(y,y')$
as 
$i\to\infty$
using the convergence
$\rho_i\to\rho_0$
in the compact-open topology. Since
$\ov\rho_i(y,y')\le(\ov\rho_i)_\tau(y,y')$,
we obtain
$$\ov\rho_0(y,y')\ge\limsup_{i\to\infty}\ov\rho_i(y,y').$$ 
\end{proof}

\subsection{Weak topology on bounded deformations}
\label{subsect:weak_topology_bounded_deform}

Let
$\cD_\rho=ds_\rho^{-1}$
be the Dirac operator associated with a deformation
$\rho$.
The compact-open topology on the set
$\Xi$
of bounded deformations (of geometry
$M$)
can be described as the topology of pointwise convergence of operators
$\cD_\rho$, $\rho\in\Xi$,
on functions with {\em finite} support. In other words,
$\rho_i\to\rho$
in the compact-open topology if and only if
$\cD_{\rho_i}(\xi)\to\cD_\rho(\xi)$
for each
$\xi\in\cH$
with finite support.

The {\em weak} topology on
$\Xi$
is defined as the topology of pointwise convergence of operators
$\cD_\rho$, $\rho\in\Xi$,
on arbitrary elements
$\xi\in\cH$.
This simultaneously defines a corresponding topology on the set 
of the associated Connes metrics
$d_\rho$.
The weak topology is finer than the compact-open one. Still, it is not clear 
whether the convergence of deformations in the former topology implies
the Hausdorff-Gromov convergence of associated Connes metrics
$d_\rho$.

\noindent
{\em Question.} Assume that the marked spectra 
$\rho_i\to\rho$
with respect to the weak topology. Is it true that the associated Connes
metrics
$d_i$
converge to
$d_0$
uniformly or at least by Hausdorff-Gromov? (See Proposition~\ref{pro:lip_topology_convergence} 
below).

\subsection{Uniform topology and the Lipshitz convergence}
\label{subsect:uniform_topology_Lipschitz}

The {\em uniform} topology on
$\Xi$
is defined as the topology of the norm convergence of operators
$\cD_\rho$, $\rho\in\Xi$.
In other words, the base of neighborhoods of a marked spectrum
$\rho_0\in\Xi$
in this topology is formed by the sets
$$U_\ep(\rho_0)=\set{\rho\in\Xi}
 {$\left|\frac{1}{\rho_0(y,y')}-\frac{1}{\rho(y,y')}\right|<\ep\quad
  \textrm{for all}\quad (y,y')\in\supp M$},$$
$\ep>0$.
Recall that the Lipschitz distance between metrics
$d$, $d'$
on
$X$
is defined as
$$|dd'|_L=\sup_{x\neq y}\left|\log\frac{d'(x,y}{d(x,y)}\right|.$$ 
The next statement follows almost immediately from definitions.

\begin{pro}\label{pro:uniform_lipschits} Let
$M=\{\cA,\cH,ds\}$
be a spectral geometry on a compact metric space
$X$.
Assume that a sequence of marked spectra
$\{\rho_i\}\sub\Xi$
of its bounded deformations converges to
$\rho_0\in\Xi$
in the uniform topology. Then for the associated Connes metrics on
$X$a
we have
$$|d_{\rho_i}d_{\rho_0}|_L\to 0\quad\textrm{as}\quad i\to\infty.$$
\qed
\end{pro}

\subsection{Pertubations of degenerations}
\label{subsect:pertubations_degenerations}

A deformation
$ds_\rho$
of a spectral geometry
$M$
is said to be {\em regular,} if
$\rho(y,y')=\ov\rho(y,y')$
for all 
$(y,y')\in\supp M$
(see sect.~\ref{subsect:bounded_deformations_connected}). 
To a given deformation
$\rho$
it canonically corresponds a regular deformation
$\ov\rho$.
A deformation
$\rho$
is said to be {\em simple,} if for every
$(y,y')$, $(z,z')\in\supp M$
the condition
$\rho(y,y')=\rho(z,z')$
implies
$(y,y')=(z,z')$
or
$(y,y')=(z',z)$.

\begin{lem}\label{lem:approximation_simple_regular} Every spectral geometry
$M=\{\cA,\cH,ds\}$
on a compact metric space
$X$
can be approximated in the uniform topology by its
simple regular deformations.
\end{lem}

\begin{proof} Given
$z\in B=\supp M$
we put
$\rho(z)=|z|$.
Obviously,
$\ov\rho=\rho$,
i.e., the geometry
$M$
itself is regular. The support
$B$
can be represented as a sequence
$B=\{z_1,z_2,\dots\}$
in such a way that
$\rho(z_{i+1})\le\rho(z_i)$
for every
$i\ge 1$.
We fix a positive sequence
$\{\ep_i\}$
satisfying the following conditions
\begin{itemize}
 \item [(1)] $\ep_i<\rho(z_i)^2/(1+\rho(z_i))$
for all
$i\ge 1$;
 \item [(2)] $\sum_{i>j}\ep_i<\ep_j$
for all 
$j\ge 1$;
 \item [(3)] if
$\rho(z_{i+1})<\rho(z_i)$,
then
$\ep_k<\rho(z_i)-\rho(z_{i+1})$
for 

$k=\min\set{j}{$\rho(z_j)=\rho(z_i)$}$.
\end{itemize}
Given
$h\in(0,1)$
and
$i\ge 1$
we put
$\rho_h(z_i)=\rho(z_i)-h\ep_i$.
Then
$\rho_h(z_i)>0$
since
$\ep_i<\rho(z_i)$.
Furthermore,
$$\frac{1}{\rho_h(z_i)}-\frac{1}{\rho(z_i)}=
   \frac{h\ep_i}{\rho(z_i)(\rho(z_i)-h\ep_i)}<h$$
by condition (1). Thus
$\rho_h\to\rho$
as
$h\to 0$
in the uniform topology. It follows from the regularity of
$\rho$
and condition~(2) that deformations
$\rho_h$
are regular for all
$h\in(0,1)$.
Conditions~(2) and (3) imply that all deformations
$\rho_h$
are simple. 
\end{proof}

\section{Dimension of spectral geometry}
\label{sect:dimension_spectral_geometry}

\subsection{Definition and properties of dimension of spectral geometries}
\label{subsect:def_properties_dimension}

The dimension of a spectral geometry is defined in \cite{Bu}
via the notion of the Dixmier trace
$\tr_\om$
(see \cite{Di},\cite{Co}). Let
$M=\{\cA,\cH,ds\}$
be a spectral geometry on a compact metric space
$X$.
Its {\em dimension} is defined as
$$\dim M:=\sup\set{p\ge 0}{$\tr_\om(|ds|^p)=\infty$}
         =\inf\set{p\ge 0}{$\tr_\om(|ds|^p)=0$}.$$
The following lemma shows that this notion is well defined.

\begin{lem}\label{lem:dimension_spectral_geom} Let
$M=\{\cA,\cH,ds\}$
be a spectral geometry on a compact metric space
$X$.
Put
$p_-=\sup\set{p\ge 0}{$\tr_\om(|ds|^p)=\infty$}$,
$p_+=\inf\set{p\ge 0}{$\tr_\om(|ds|^p)=0$}$.
Then
$p_-=p_+$. 
\end{lem}

\begin{proof} Obviously,
$p_-\le p_+$.
Suppose that 
$p_-<p_+$.
Then
$$0<\tr_\om(|ds|^p)<\infty$$
for every
$p\in(p_-,p_+)$.
Let
$\mu_0\ge\mu_1\ge\dots$
be the sequence of eigenvalues of
$|ds|$,
written in nondecreasing order according to multiplicity. We choose 
$p$, $p'\in(p_-,p_+)$, $p'>p$,
and fix 
$\ep>0$.
There is
$n_0>0$
such that
$\mu_n^{p'-p}<\ep$
for all
$n\ge n_0$
and, furthermore,
$\si_N(|ds|^p)=\sum_{n=0}^{N-1}\mu_n^p\le C\log N$
for all sufficiently large 
$N$,
where the constant
$C$
does not depend on
$N$.
Then
$$\si_N(|ds|^{p'})\le\sum_{n=0}^{n_0-1}\mu_n^{p'}+\ep\sum_{n=n_0}^{N-1}\mu_n^p\le\ep(1+C)\log N$$
for all sufficiently large
$N$.
Thus
$\tr_\om(|ds|^{p'})=0$.
This contradicts the assumption
$p'<p_+$.
\end{proof}

Let
$E_t$
be the maximal subspace in
$\cH$
with
$\|ds^{-1}|E_t\|\le t^{-1}$.
In the case when the multiplicity function
$m\equiv 1$,
we have
$\dim E_t=|B_t|$.
The following proposition is proved in \cite{Bu}.

\begin{pro}\label{pro:various_dimensions} The dimension
$\dim M$
coincides with
$$\dim_2M=\limsup_{t\to 0}\frac{\log\dim E_t}{\log1/t}\ \textrm{and}\  
  \dim_3M=\inf\set{p\ge 0}{$\tr_\om(|ds|^p)<\infty$}.$$
\qed     
\end{pro}

Furthermore, we have

\begin{lem}\label{lem:finite_trace} For every
$p>\dim M$
the trace of the operator
$|ds|^p$
is finite,
$$\tr_\om(|ds|^p)<\infty.$$
\end{lem}

\begin{proof} As usual, let
$\mu_0\ge\mu_1\ge\dots$
be the sequence of eigenvalues of the operator
$|ds|$
written in nondecreasing order according to multiplicity. By
Proposition~\ref{pro:various_dimensions}, we have
$$\dim M=\limsup_{t\to 0}\frac{\log\dim E_t}{\log t^{-1}}.$$
Take
$\ep>0$
such that
$d=\dim M<p-\ep$.
Then for every sufficiently small
$t>0$
we have
$\dim E_t\le t^{-(d+\ep)}$,
where
$\dim E_t=\sum_{(y,y')\in B_t}m(y,y')$
coincides with the number of eigenvalues (with multiplicity) of
$|ds|$,
which
$\ge t$.
Therefore, if
$\mu_n=t$,
then
$n\le t^{-(d+\ep)}+C$
and
$\mu_n^p\le (n-C)^{-p/(d+\ep)}$,
where the constant
$C$
does not depend on
$n$.
Hence, the claim. 
\end{proof}

\subsection{Dixmier trace and Hausdorff measures}
\label{subsect:dixmier_trace_hausdorff_measures}

Let
$M=\{\cA,\cH,ds\}$
be a spectral geometry on
$X$, $B\sub X^2\sm\De$
its support. As usual, for 
$t>0$
we put
$B_t=\set{b\in B}{$|b|\ge t$}$, $Y_t\sub X$
the projection of
$B_t$
on the factor
$X$, $\de(t)$
the bounding radius of
$Y_t$
(see sect.~\ref{subsect:bounded_deformations_hausdorff_gromov}).
The {\em relative bounding radius} of the support
$B$
of the geometry
$M$
is defined as
$$\nu=\nu(B)=\limsup_{t\to 0}\de(t)/t.$$
In what follows we assume that
$\nu<\infty$.
This condition is fulfilled if, for example,
$D=\dens(M)>2$
(see Lemma~\ref{lem:bounding_radius}). We use notation
$H^p(X)$
for the 
$p$-dimensional
Hausdorff measure of
$X$.
Recall that
$H^p(X)=\lim_{t\to 0}H_t^p(X)$,
where
$$H_t^p(X)=\inf\sum_{\diam F\le t}(\diam F)^p$$
and the infimum is taken over all covers of
$X$
by nonempty closed subsets of diameter
$\le t$.

\begin{thm}\label{thm:hausdorff_measure_above} Given
$p>0$,
we have
$$(2\nu^p)\tr_\om(|ds|^p)\ge H^p(X)$$
(in the case
$\nu=0$, $H^p(X)>0$
it means that 
$\tr_\om(|ds|^p)=\infty$).
\end{thm}

\begin{proof} We can assume that
$H^p(X)>0$
since otherwise there is nothing to prove.

(1) The set 
$Y_t$
is a 
$\de(t)$-net
in
$X$.
Thus
$$\sum_{y\in Y_t}\left(\diam D_{\de(t)}(y)\right)^p\ge H_{2\de(t)}^p(X).$$
On the other hand,
$$\sum_{y\in Y_t}\left(\diam D_{\de(t)}(y)\right)^p\le (2\de(t))^p|Y_t|,$$
whence
$$|Y_t|\ge\frac{H_{2\de(t)}^p(X)}{(2\de(t))^p}.$$ 
Fix 
$\ep>0$.
Then
$\de(t)/t\le\nu+\ep$
for all sufficiently small
$t>0$.
Furthermore, we can assume that
$H_{2\de(t)}^p\ge H^p(X)-\ep$
(if
$H^p(X)=\infty$,
then
$H_{2\de(t)}^p(X)\ge 1/\ep$).
For such
$t$
we have
$$|Y_t|\ge\frac{H^p(X)-\ep}{(2(\nu+\ep))^pt^p}.$$

(2) We consider first the case when the marked spectrum of
$ds$
is simple (see sect.~\ref{subsect:pertubations_degenerations}).
One can also assume that the multiplicity function
$m:B\to\N$
is constant and equals one, since otherwise the required inequality
becomes only stronger. Therefore, the multiplicity of every eigenvalue of
$|ds|$
equals 2. Let
$\mu=\{\mu_1=\mu_2>\mu_3=\mu_4>\dots\}$
be the sequence of the eigenvalues of
$|ds|$
written in nonincreasing order according to multiplicity. We put
$\mu_{k-1}=\mu_k=t$.
Then for all sufficiently large
$k$,
say
$k\ge k_0$,
we have due to (1)
$$k=|B_t|\ge|Y_t|\ge\frac{H^p(X)-\ep}{(2(\nu+\ep))^pt^p}.$$
Hence
$$\mu_{k-1}^p=\mu_k^p\ge\frac{H^p(X)-\ep}{(2(\nu+\ep))^pk}.$$
Recalling the definition of the Dixmier trace, we immediately obtain
$$\tr_\om(|ds|^p)\ge\frac{1}{(2\nu)^p}H^p(X),$$
and in the case of simple spectrum the theorem is proved.

(3) Now, we consider the general case. By Lemma~\ref{lem:approximation_simple_regular},
the operator
$ds$
can be approximated in the uniform topology by simple regular deformations.
For every such a deformation
$\rho$
we have already proved that
$$(2\nu)^p\tr_\om(|ds|^p)\ge H^p((X,d_\rho)),$$
where
$d_\rho$
is the Connes metric on
$X$
associated with the deformation. By Proposition~\ref{pro:uniform_lipschits},
these Connes metrics approximate the initial metric of
$X$
in the Lipschitz metric. Thus
$H^p((X,d_\rho))\to H^p(X)$.
On the other hand, the convergence
$\rho\to\mu$
in the uniform topology, obviously, implies that
$\tr_\om(|ds_\rho|^p)\to\tr_\om(|ds|^p)$.
Thus to complete the proof, it suffices to show that 
$\limsup_{\rho\to\mu}\nu_\rho\le\nu$.

Fix 
$\ep>0$.
Then for all deformations
$\rho$
sufficiently close to
$\mu$
we have
$1-\ep\le|xx'|_\rho/|xx'|\le 1+\ep$
for each
$x$, $x'\in X$.
Thus
$B_t\sub B_{\rho,(1-\ep)t}$
and
$Y_t\sub Y_{\rho,(1-\ep)t}$,
where
$B_{\rho,t}=\set{b\in B}{$|b|_\rho\ge t$}$
and
$Y_{\rho,t}$
is the projection of
$B_{\rho,t}\sub X^2$
on the factor
$X$.
Then for the bounding radius we have
$$\de(t)=\de(Y_t)\ge\de(Y_{\rho,(1-\ep)t})\ge(1+\ep)^{-1}\de_\rho(Y_{\rho,(1-\ep)t})
        =\de_\rho((1-\ep)t)/(1+\ep).$$
Consequently
$$\nu=\limsup_{t\to 0}\de(t)/t\ge\nu_\rho(1-\ep)/(1+\ep),$$
which implies the required inequality. Note that similar argument shows that
actually
$\nu_\rho\to\nu$.
\end{proof}

\begin{rem}\label{rem:relative_bounding_factor} Appearence of the factor
$\nu$
in the inequality above, which is closely related to the density of
spectral geometry, is not incidental. There is a similar factor in the 
Connes-Sullivan formula \cite[IV, Theorem~17]{Co} and other formulae from
\cite{Co}.
\end{rem}

Here we consider this factor in more details. If the density of a spectral 
geometry
$D=\dens M>2$,
then by Lemma~\ref{lem:bounding_radius},
$\nu\le D/(D-2)$. 
One can give more precise estimate for sufficiently ``regular'' spaces.

A metric space 
$X$
is called
{\em $R$-uniformly perfect,}
$R\ge 1$,
if there is
$\ep>0$
such that for every
$x\in X$
the set 
$D_r(x)\sm D_{r/R}(x)$
is not empty for all
$0<r\le\ep$.

The following lemma is a version of Lemma~\ref{lem:bounding_radius}.

\begin{lem}\label{lem:loc_bound_radius_estimate} Assume that a metric space
$X$ 
is
$R$-uniformly perfect, $R\ge 1$,
a subset
$B\sub X^2\sm\De$
is symmetric and has local density
$\dens_{\loc}(B)=D>2$.
Given
$t>0$
let
$Y_t\sub X$
be the projection on
$X$
of
$B_t=\set{(y,y')\in B}{$|yy'|\ge t$}$, $\de(t)$
the bounding radius of
$Y_t$.
Then
$$\nu(B)=\limsup_{t\to 0}\de(t)/t\le R/(D-2).$$
In particular, 
$\nu(B)=0$,
if
$\dens_{\loc}(B)=\infty$. 
\end{lem}

\begin{proof} We can assume that
$D<\infty$.
Suppose that
$\nu(B)>R/(D-2)$
and fix 
$\ep\in (0,D-2-R/\nu)$.
Then
$\nu>R/(D-2-\ep)$,
and there is an arbitrarily small
$t>0$
with

\begin{itemize}
 \item [(1)] $\de(t)/t\ge R/(D-2-\ep)$;
 \item [(2)] for every pair
$(x,x')\in X^2\sm\De$
there is a pair
$(y,y')\in B$
with
$|xy|$, $|x'y'|\le\al|xx'|$,
where
$\al=1/(D-\ep)<1/2$.
\end{itemize}

Then for 
$r=(D-\ep)\de(t)$
and
$r'=(D-\ep)t/(D-2-\ep)$
we have
$r/r'\ge R$
by (1). Now we chose a point
$x\in X$
with
$|xY_t|=\de(t)$.
Since the space is uniformly perfect, there is point
$x'\in X$
with
$$\frac{(D-\ep)t}{D-2-\ep}\le|xx'|<(D-\ep)\de(t).$$
We approximate the pair
$(x,x')$
according (2) by a pair
$(y,y')\in B$
with
$|xy|$, $|x'y'|\le\al|xx'|$.
Then
$|yy'|\ge(1-\al)|xx'|\ge t$,
thus
$(y,y')\in B_t$
and
$y\in Y_t$.
Furthermore,
$|xy|\le\al|xx'|<\de(t)$,
which is a contradiction with the choice of
$x$.
\end{proof}

The next statement directly follows from Theorem~\ref{thm:hausdorff_measure_above}.

\begin{cor}\label{cor:dim_hausdorff_dim_below} Let
$M$
be a spectral geometry on a compact metric space
$X$,
and assume that 
$\nu(B)<\infty$.
Then
$$\dim M\ge\dim_HX,$$
where
$\dim_HX$
is the Hausdorff dimension of
$X$.
\qed 
\end{cor}

\subsection{Weak $L^p$-topology on bounded deformations}
\label{subsect:weak_Lp_topology_bounded_deformations}

Let
$\{\cA,\cH,ds\}$
be a spectral geometry on a compact metric space
$X$.
For
$p>0$
we denote by
$L^p(B)$
the Banach space of sections
$\xi:B\to\C^m$
satisfying the condition
$$\|\xi\|^p=\sum_{(y,y')\in B}\|\xi(y,y')\|^2<\infty,$$
where
$B\sub X^2\sm\De$
is the support of
$M$, $m:B\to\N$
the multiplicity function (cp. sect.~\ref{sect:axioms_spectral}).
Given a bounded deformation
$\rho\in\Xi$
of
$M$
(see sect.~\ref{sect:bounded_deformations}), the operator
$\cD_\rho=ds_\rho^{-1}$
is defined on a dense subset in
$L^p(B)$.
The {\em weak $L^p$-topology} on the set 
$\Xi$
of bounded deformations is defined as the topology of pointwise 
convergence of operators
$\cD_\rho$, $\rho\in\Xi$,
on arbitrary elements
$\xi\in L^p(B)$.
For 
$p=2$,
this definition coincides with that one given in 
sect.~\ref{subsect:weak_topology_bounded_deform}.

\begin{pro}\label{pro:lip_topology_convergence} Let
$M=\{\cA,\cH,ds\}$
be a spectral geometry on a compact metric space
$X$.
Assume that a sequence of marked spectra
$\{\rho_i\}\sub\Xi$
converges in the weak
$L^p$-topology
to
$\rho_0\in\Xi$.
If
$p>\dim M$,
then the associated Connes metrics
$d_{\rho_i}$
on
$X$
converge in the Lipschitz topology to
$d_{\rho_0}$,
$$|d_{\rho_i}d_{\rho_0}|\to 0\quad\textrm{as}\quad i\to\infty.$$
In particular, if
$\dim M<2$,
then the convergence of marked spectra in the weak topology implies
the convergence of the associated Connes metrics in the Lipschitz topology.
\end{pro}

\begin{proof} Using Lemma~\ref{lem:finite_trace} and the fact that
every deformation
$\rho\in\Xi$
is bounded, we obtain
$$\sum_{(y,y')\in B}(\rho(y,y'))^p\le C\sum_{(y,y')\in B}|yy'|^p
  \le C\tr(|ds|^p)<\infty.$$
It is easy to construct a section
$\xi_0:B\to\C^m$
with
$\|\xi_0(y,y')\|=\rho_0(y,y')$
for all
$(y,y')\in B$.
Then
$\xi_0\in L^p(B)$
and, by the condition,
$$\sum_{(y,y')\in B}\left|\frac{1}{\rho_i(y,y')}-\frac{1}{\rho_0(y,y')}\right| 
  \|\xi_0(y,y')\|^p\to 0\quad\textrm{as}\quad i\to\infty.$$
By the choice of
$\xi_0$
this implies that for every
$\ep>0$
we have
$$1-\ep\le\frac{\rho_0(y,y')}{\rho_i(y,y')}\le 1+\ep$$
for all
$(y,y')\in B$
and all sufficiently large
$i$.
It follows that for such
$i$
we have
$$(1-\ep)d_{\rho_i}(x,x')\le d_{\rho_0}(x,x')\le(1+\ep)d_{\rho_i}(x,x')$$
for every
$x$, $x'\in X$.
This means that
$$|d_{\rho_i}d_{\rho_0}|_L\to 0\quad\textrm{as}\quad i\to\infty.$$ 
\end{proof}

\subsection{Estimate from above for the spectral dimension}
\label{subsect:estimate_above_sp_dimension}

The {\em spectral dimension} of a compact metric space 
$X$
is defined as
$$\dim_S(X):=\liminf\set{\dim M}{$\dens_{\loc}(M)\to\infty$},$$
where the lower limit is taken over all spectral geometries on
$X$,
whose local density unboundedly increases. It is easy to see that the 
spectral dimension is a biLipschitz invariant of
$X$.
By Lemma~\ref{lem:loc_bound_radius_estimate}, the relative bounding
radius of
$M$
is finite, if 
$\dens_{\loc}(M)>2$.
Thus
$\dim_SX\ge\dim_HX$
by Corollary~\ref{cor:dim_hausdorff_dim_below}.

For sufficiently homogeneous spaces, the spectral dimension 
coincides with the Hausdorff dimension. A metric space
$X$
is said to be 
{\em $q$-quasi-homogeneous,} 
($q$-q.h. for brevity),
$q>0$, 
if for every
$\de>0$
there exists a
$\de$-net $Y_\de\sub X$
such that the ball
$D_\rho(x)$
contains at most 
$C(\rho/\de)^q$
points of
$Y_\de$
for every
$\rho>0$
and 
$x\in X$,
where the constant
$C>0$
does not depend on
$\de$, $\rho$
and
$x$.

\begin{thm}\label{thm:qqh_space_dimension} Assume that a compact metric space
$X$
is
$q$-q.h.
Then there is a spectral geometry
$M$
on
$X$
with infinite local density and the dimension
$\dim M\le q$.
In particular,
$\dim_SX\le q$.
\end{thm}

The proof is based on the following proposition.

\begin{pro}\label{pro:qqh_estimate_above} Assume that a compact metric space
$X$
is
$q$-q.h.
Then there is a symmetric subset
$B\sub X^2\sm\De$
having the infinite local density such that the set 
$B_t=\set{b\in B}{$|b|\ge t$}$
contains at most 
$C(\log 1/t)^{2q+1}/t^q$
elements for every (sufficiently small)
$t>0$,
where the constant
$C>0$
does not depend on
$t$.
\end{pro}

\begin{proof} We put
$\de_k=2^{-k}$
and for every
$k\ge 5$
consider a 
$\de_k$-net $Y_k=Y_{\de_k}$
from the definition of quasi-homogeneity. Due to that property, we have
$|Y_k|\le C\de_k^{-q}$.
Let
$B=\cup_kB(k)$,
where
$B(k)=\set{(y,y')\in Y_k^2\sm\De}{$|yy'|\le R_k$}$,
$R_k=(k-1)\de_{k-1}$.
The set 
$B$
is, obviously, symmetric and by Proposition~\ref{pro:large_density}
has the infinite local density. Let
$t>0$.
If
$(y,y')\in B_t$,
then
$(y,y')\in B(k)$
with
$R_k\ge|yy'|\ge t$.
Thus
$(k-1)/2^{k-1}\ge t$.
Let
$k_0$
be the maximal integer satisfying this inequality. Then
$$|B_t|\le\sum_{k=5}^{k_0}\sum_{y\in Y_k}|D_{R_k}(y)\cap Y_k|
   \le C\sum_{k=5}^{k_0}|Y_k|(R_k/\de_k)^q\le Ck_0^{q+1}2^{(k_0+1)q}.$$
Using the condition
$k_0-1\ge t2^{k_0-1}$,
we obtain that
$|B_t|\le C'(\log 1/t)^{2q+1}/t^q$,
where the constant
$C'>0$
does not depend on
$t$. 
\end{proof}

\begin{proof}[Proof of Theorem~\ref{thm:qqh_space_dimension}]
We consider the spectral geometry
$M=M(B,1)$
on
$X$,
where the set 
$B$
is constructed in Proposition~\ref{pro:qqh_estimate_above}.
For the subspace
$E_t\sub\cH$
and every
$t>0$, 
we obviously have 
$\dim E_t=|B_t|$.
Computing the dimension of
$M$
with help of Proposition~\ref{pro:various_dimensions} and using
the estimate from Proposition~\ref{pro:qqh_estimate_above},
we obtain
$\dim M\le q$.
\end{proof}

Every compact Riemannian manifold 
$X$
is obviously
$q$-q.h.
with
$q=\dim X$.
Thus using Theorem~\ref{thm:qqh_space_dimension}
and Corollary~\ref{cor:dim_hausdorff_dim_below}, we obtain

\begin{cor}\label{cor:spectral_dim_riemannian} Every compact Riemannian
manifold
$X$
possesses a spectral geometry
$M$
with infinite local density and the dimension
$\dim M=\dim X$.
In particular,
$\dim_SX=\dim X$.
\qed 
\end{cor}

We give another example of compact Hausdorff space, whose
spectral dimension coincides with the Hausdorff one.

Given
$n\in\N$
and
$p>1$,
the set 
$\cP_{n,p}=\{0,\dots,n\}^{\N}$
with metric
$$|xx'|=\frac{1}{p^k},$$
where
$k=\min\set{N}{$x({N+1})\neq (x'(N+1)$}$,
is a compact Cantor space of diameter
$\diam\cP_{n,p}=1$.

\begin{lem}\label{lem:qqh_cantor} The compact space
$\cP_{n,p}$
is
$q$-quasi-homogeneous
with
$q=\log_p(n+1)$.
\end{lem}

\begin{proof} For every
$N\in\N$,
the set 
$$\cP_{n,p}^N=\set{x\in\cP_{n,p}}{$x(k)=x(N)\ \textrm{for all}\ k\ge N$}$$
is a 
$\de$-net 
in
$\cP_{n,p}$
with
$\de=p^{-N}$
and contains
$(n+1)^N$
elements. Therefore, for
$Y_\de=\cP_{n,p}^N$  
we have
$|Y_\de|=(n+1)^N=\de^{-q}$.
Given
$\rho>0$
and
$x\in\cP_{n,p}$,
the ball
$B_\rho(x)$
coincides with
$\set{x'\in\cP_{n,p}}{$p^{-Q}\le\rho$}$,
where
$Q=Q(x')$
and
$Q+1$
is the minimal integer
$k$
with
$x'(k)\neq x(k)$.
Thus for the minimal
$Q$
satisfying
$p^{-Q}\le\rho$,
we have 
$$|B_\rho(x)\cap Y_\de|\le(n+1)^{N-Q}\le\de^{-q}p^{-qQ}\le C(\frac{\rho}{\de})^q,$$
where the constant
$C>0$
depends only on
$p$
and
$q$.
\end{proof}

The following statement is well known (at least for 
$p=2$),
and we give its proof for sake of completeness.

\begin{lem}\label{lem:dim_hausdorff_exa_below} 
$\dim_H\cP_{n,p}\ge\log_p(n+1)$.
\end{lem}

\begin{proof} We introduce a probability measure on
$\cP_{n,p}$
defined by the condition that the numbers
$\{0,\dots,n\}$
appear as
$x(k)$
with probability
$1/(n+1)$
for every
$x\in\cP_{n,p}$, $k\in\N$.
In other words, for every
$x\in\cP_{n,p}$, $N\in\N$
we have
$$\mes(B_{p^{-N}}(x))=\frac{1}{(n+1)^N}.$$
Thus for 
$q=\log_p(n+1)$
we have
$\mes(B_{p^{-N}}(x))=(\diam (B_{p^{-N}}(x)))^q=p^{-qN}$.
Given a closed subset
$F\sub\cP_{n,p}$,
there are 
$x\in\cP_{n,p}$, $N\in\N$
such that
$F\sub B_{p^{-N}}(x)$
and 
$\diam F=\diam (B_{p^{-N}}(x))$.

Indeed,
$F$
is compact, thus there exist 
$x$, $x'\in F$
with
$$|xx'|=\diam (F)=p^{-N}$$
for some 
$N\in\N$.
Hence
$F\sub B_{p^{-N}}(x)$
and
$\diam (F)=\diam (B_{p^{-N}}(x))$.

Therefore,
$$\mes(F)\le(\diam(F))^q$$
for every closed
$F\sub\cP_{n,p}$.
Thus given a covering 
$\{F\}$
of
$\cP_{n,p}$
by closed subsets, we have
$$\sum(\diam F)^q\ge\sum\mes(F)\ge\mes(\cP_{n,p})=1.$$
For the 
$q$-dimensional
Hausdorff measure this gives
$H^q(\cP_{n,p})\ge 1$,
and
$\dim_H\cP_{n,p}\ge\log_p(n+1)$.
\end{proof}

From Corollary~\ref{cor:dim_hausdorff_dim_below}, Theorem~\ref{thm:qqh_space_dimension} 
and Lemmas~\ref{lem:qqh_cantor} and \ref{lem:dim_hausdorff_exa_below} we obtain

\begin{cor}\label{cor:spgeometry_P} For every
$n\in\N$, $p>1$,
there exists a spectral geometry
$M$
with infinite local density and the dimension 
$\dim M=\log_p(n+1)=\dim_H\cP_{n,p}$
on the compact metric space 
$\cP_{n,p}$.
In particular,
$\dim_S\cP_{n,p}=\dim_H\cP_{n,p}$.
\qed
\end{cor}

\subsection{Comparison with other definitions of dimension}
\label{subsect:Comparison_def_dimension}

A number of definitions of dimension suitable for subsets of metric spaces is
known (see, for example \cite{KT}, \cite{Tr}), especially in the literture
on fractal geometry \cite{Fa1}, \cite{Fa2}, which give different results for 
sufficiently pathological spaces. We mention here besides the Hausdorff's also
packing dimension, upper and lower box dimension, the Minkowski-Bouligand
dimension \cite{Bo} (the last one is defined for subsets of Euclidean spaces).
It is well known that the upper box dimension dominates all dimensions from
this list. Thus we compare the spectral dimension
$\dim_S$
only with the upper box dimension
$\ov{\dim}_B$, 
which is defined as follows.

Let
$N_t(X)$
be the least number of elements, which is need to cover
$X$
by sets of diameter
$\le 2t$.
We put 
$$\ov{\dim}_BX=\limsup_{t\to 0}\frac{\log N_t(X)}{\log t^{-1}}.$$
In \cite{KT}, this quantity is called the {\em upper metric dimension} of
$X$.

\begin{thm}\label{thm:spectral_above_upperbox} For each compact metric space
$X$
it holds
$$\dim_SX\ge\ov{\dim}_BX.$$
\end{thm}

\begin{proof} It is easy to see that
$N_t(X)$
is at most the maximal number of
$t$-separated
points in
$X$ 
and is at least the maximal number of
$2t$-separated
points in
$X$
(see \cite[Theorem IV]{KT}). Thus in the definition of the upper box dimension,
one can take as
$N_t$
the maximal number of 
$t$-separated
points in 
$X$.

We fix 
$\ep>0$
and
$D\ge 4$. 
Since
$$\dim_SX\ge\inf\set{\dim M}{$\dens_{\loc}(M)\ge D+\ep$},$$ 
there exists a spectral geometry
$M$
on
$X$
with
$\dens_{\loc}(M)\ge D+\ep$, 
for which
$\dim_SX+\ep\ge\dim M$.
Using the definition of the local density
$\dens_{\loc}(M)$,
we find 
$t>0$
such that for each
$x$, $x'\in X$
with
$0<|xx'|\le t$
there is
$b=(y,y')\in\supp M$
with
$\max\{|xy|,|x'y'|\}\le|xx'|/D$.

The set of
$t/2$-isolated
points of
$X$
(i.e., separated from any other point by the distance
$>t/2$)
is finite because 
$X$
is compact. Removing this set from
$X$,
we change neither the spectral nor the upper box dimension of
$X$.
Thus in what follows, we assume that
$X$ 
contains no
$t/2$-isolated
point.

Any subset in
$X$
of diameter
$\le t/10$,
separated from its complement by the distance
$>t/2$,
is called a
{\em $t$-claster}.
Obviously, the number of
$t$-clasters 
in
$X$
is finite. Let
$\de_0$
be the minimum of their diameters. Then
$\de_0>0$
since
$X$
contains no
$t/2$-isometry
point.

For each
$\de\in(0,\de_0/2)$
there is
$\tau\in(0,\de)$
with
$$\frac{\log N_\tau(X)}{\log\tau^{-1}}\ge\ov{\dim}_BX-\ep.$$
Let
$A_\tau$
be a subset in
$X$,
consisting of the maximal number of
$\tau$-separated
points,
$|A_\tau|=N_\tau(X)$.
For each point
$a\in A_\tau$,
we associate the pair
$z_a=(a,a')\in X^2\sm\De$,
where
$a'\in A_\tau$
be (some) closest to
$a$
point. Then
$|aa'|\ge\tau$.
We show that
$|aa'|\le t$.

Assume that
$|aa'|>t$.
There is
$x\in X$
different from
$a$
with
$|xa|\le t/2$,
since there is no
$t/2$-isolated
point. Then
$|xa'|>t/2\ge 5\de_0\ge 5\tau$.
Moreover, by the definition of
$\de_0$,
one can choose
$x$
so that
$|xa|\ge\de_0/2\ge\tau$.
However, this contradicts the maximality of
$A_\tau$.
Therefore,
$|aa'|\le t$.

For different
$a$, $a'\in A_\tau$
(unordered) pairs
$z_a$, $z_{a'}$
coincide or
$|z_az_{a}|\ge\tau$.
Hence, we obtained at least
$N_\tau(X)/2$
such pairs. Each pair
$z_a=(a,a')$
can be approximated by
$b_a=(y,y')\in\supp M$
so that
$\max\{|ay|,|a'y'|\}\le|aa'|/D$,
because
$|aa'|\le t$.
Since
$D\ge 4$,
for different pairs such points are different and
$|yy'|\ge(1-2/D)|aa'|\ge\tau/2$.
Thus for the spectral geometry
$M$
we have
$\dim E_{\tau/2}\ge N_\tau(X)/2$
and
$$\frac{\log\dim E_{\tau/2}}{\log(\tau/2)^{-1}}
  \ge(\ov{\dim}_BX-\ep)\frac{1-\log 2/\log N_\tau(X)}{1+\log 2/\log\tau^{-1}}.$$
Taking
$\de\to 0$,
we obtain
$$\dim_SX+\ep\ge\dim M\ge\ov{\dim}_BX-\ep.$$
Now, the required inequality follows since
$\ep$
is arbitrary.
\end{proof}

\begin{rem}\label{rem:more_precise} One can see from the proof that 
the following more precise inequality
$$\ov{\dim}_BX\le\inf\set{\dim M}{$\dens_{\loc}(M)\ge 4$}$$
holds.
\end{rem}

\end{document}